\UseRawInputEncoding 
\documentclass[12pt]{amsart}       
\usepackage{txfonts}
\usepackage{algorithm}
\usepackage{algorithmic}
\usepackage{amssymb}
\usepackage{eucal}
\usepackage{graphicx}
\usepackage{amsmath}
\usepackage{amscd}
\usepackage[all]{xy}           
\usepackage{tikz}
\usepackage{tikz-qtree}
\usepackage{amsfonts,latexsym}
\usepackage{xspace}
\usepackage{epsfig}
\usepackage{float}
\usepackage{color}
\usepackage{fancybox}
\usepackage{colordvi}
\usepackage{multicol}
\usepackage{colordvi}
\usepackage{wasysym}
\usepackage{makecell} 
\usepackage[numbers,sort&compress]{natbib}
\usepackage[active]{srcltx} 
\usepackage{CJK}
\ifpdf
  \usepackage[colorlinks,final,backref=page,hyperindex]{hyperref}
\else
  \usepackage[colorlinks,final,backref=page,hyperindex,hypertex]{hyperref}
\fi

\newcommand{\nc}{\newcommand}
\newcommand{\delete}[1]{}
\nc{\bfk}{{\bf k}}

\nc{\mlabel}[1]{\label{#1}}  
\nc{\mcite}[1]{\cite{#1}}  
\nc{\mref}[1]{\ref{#1}}  
\nc{\mbibitem}[1]{\bibitem{#1}} 

\delete{
\nc{\mlabel}[1]{\label{#1}  
{\hfill \hspace{1cm}{\small\tt{{\ }\hfill(#1)}}}}
\nc{\mcite}[1]{\cite{#1}{\small{\tt{{\ }(#1)}}}}  
\nc{\mref}[1]{\ref{#1}{{\tt{{\ }(#1)}}}}  
\nc{\mbibitem}[1]{\bibitem[\bf #1]{#1}} 
}

\newtheorem{theorem}{Theorem}[section]

\newtheorem{lemma}[theorem]{Lemma}

\newtheorem{corollary}[theorem]{Corollary}
\theoremstyle{definition}
\newtheorem{definition}[theorem]{Definition}

\newtheorem{conjecture}{Conjecture}

\newtheorem{tempex}[theorem]{Example}
\newtheorem{tempexs}[theorem]{Examples}
\newtheorem{temprmk}[theorem]{Remark}
\newtheorem{tempexer}{Exercise}[section]
\newenvironment{example}{\begin{tempex}\rm}{\end{tempex}}

\nc{\tred}[1]{\textcolor{red}{#1}} \nc{\tgreen}[1]{\textcolor{green}{#1}}
\nc{\tblue}[1]{\textcolor{blue}{#1}} \nc{\tpurple}[1]{\textcolor{purple}{#1}}

\nc{\hu}[1]{\tpurple{\underline{Hu:}#1 }}
\nc{\xing}[1]{\tblue{\underline{Xing:}#1 }}

\nc{\GS}{Gr\"obner-Shirshov\xspace}
\nc{\gsb}{Gr\"{o}bner-Shirshov basis\xspace}
\nc{\gsbs}{Gr\"{o}bner-Shirshov bases\xspace}
\nc\olie{operated Lie algebra\xspace}
\nc\olies{operated Lie algebras\xspace}
\nc\bfone{\mathbf{1}}

\nc\nas[1]{{#1}^\ast}
\nc\alsw[1]{{\rm ALSW}(#1)}
\nc\nlsw[1]{{\rm NLSW}(#1)}
\nc\clie[1]{[#1]}
\nc\lbar[1]{\overline{#1}}
\nc\suba[1]{|_{#1}}
\nc\coplie[1]{\bfk\nlsbw{#1}}
\nc\coplieo[2]{\bfk{\rm NLSBW}_{#2}(#1)}
\nc\lb[1]{\left[#1\right]}\nc\dt[1]{{t^{(#1)}}}
\nc{\Irr}{\mathrm{Irr}}
\nc\blw[1]{\lfloor#1\rfloor}
\nc\plie[1]{\mathfrak{S}(#1)}
\nc\plien[1]{\mathcal{N}(#1)}
\nc{\lc}{\lfloor} \nc{\rc}{\rfloor}
\nc\Id{\rm Id}\nc\sopm[1]{\mathfrak{S}^\star(#1)}

\nc\ordc{>_{{\rm Dl}}} \nc\ordqc{\geq_{{\rm Dl}}}
\nc\ord{>_{{\rm IM }}}\nc\ordq{\geq _{\rm IM}}
\nc\ordd{>_{{\rm IM}}}\nc\ordqd{\geq_{{\rm IM}}}
\nc\ordb{>_{{\rm IM}}}\nc\ordqb{\geq_{{\rm IM}}}
\nc\alsbw[1]{{\rm ALSBW}_{\ordq}(#1)} \nc\nlsbw[1]{{\rm NLSBW}_{\ordq}(#1)}
\nc\alsbwo[2]{{\rm ALSBW}_{#2}(#1)} \nc\nlsbwo[2]{{\rm NLSBW}_{#2}(#1)}\nc\bws[1]{{\lfloor#1\rfloor}}\nc\oplie{{\rm OLie}(X)}
\nc{\dep}{{\rm dep}}
\nc\ordt{\geq_{\rm dt}}
\nc\ordtl{>_{\rm dt}}

\begin{document}

\title[On the Greedoid Tutte Polynomial for simple rooted graphs]{On the Greedoid Tutte Polynomial for simple rooted graphs}

\author{Jianxuan Luo, Tingzeng Wu$^{*}$ and Hong-Jian Lai}\thanks{*Corresponding author}

\address{School of Mathematical Sciences, Xiamen University, Xiamen, Fujian 361005, China}
\email{a3566293588@163.com}

\address{School of Mathematics and Statistics, Qinghai Minzu University, Xining, Qinghai 810007, China}
\email{mathtzwu@163.com}

\address{School of Mathematics and System Sciences, Guangdong Polytechnic Normal University, Guangzhou 510665, China}
\email{hjlai2015@hotmail.com}

\date{\today}

\begin{abstract}
The Tutte polynomial, through its rank-generating formula, provides a unified framework for describing various combinatorial structures of graphs and matroids, 
and serves as an important polynomial invariant connecting graph theory, matroid theory, and their related applications.
Let $G=(V(G),E(G),r(G))$ be a simple rooted graph, where $V(G)$ is the vertex set, $E(G)$ is the edge set, and $r(G)$ is the root. 
Let $\Gamma(G)$ be the greedoid induced by the rooted graph $G$, and let its greedoid Tutte polynomial be denoted by $T(\Gamma(G);x,y)$.
Let $r=r(G)$, $L_r(G)=\{v\in V(G)\setminus\{r\}: rv\in E(G),\ d_G(v)=1\}$, and $\ell_r(G)=|L_r(G)|$. 
Gordon and McMahon proposed the following conjecture: for rooted graphs, $\operatorname{ord}_{x}T(\Gamma(G);x,y)=\ell_r(G)$,
where $\operatorname{ord}_{x}T(\Gamma(G);x,y)=\max\{a\in\mathbb Z_{\geq0}:x^a\mid T(\Gamma(G);x,y)\}$, that is, 
the highest power of $x$ dividing $T(\Gamma(G);x,y)$ is equal to the number of leaf vertices adjacent to the root. 
In this paper, we proved that this conjecture holds for all simple rooted graphs.
\end{abstract}

\makeatletter
\@namedef{subjclassname@2020}{\textup{2020} Mathematics Subject Classification}
\makeatother
\subjclass[2020]{
    05B35; 05C31; 05C05\underline{}
}

\keywords{ Greedoid; Tutte polynomial; Greedoid polynomial; Rooted graph}

\maketitle

\tableofcontents

\setcounter{section}{0}

\allowdisplaybreaks

\section{Introduction}
In this paper, unless otherwise specified, all graphs are allowed to have loops and multiple edges. 
A graph without loops and multiple edges is called a simple graph. Let $G=(V(G),E(G))$ be a graph with vertex set $V(G)$ and edge set $E(G)$. 
For $U\subseteq V(G)$, let $G[U]$ denote the subgraph of $G$ induced by $U$. For $v\in V(G)$, 
let $N_G(v)$ be the neighborhood of the vertex $v$ in $G$, and $d_G(v)$ be the degree of $v$, 
a vertex satisfying $d_G(v)=1$ is called a leaf vertex. 
A path in $G$ refers to a simple path with no repeated vertices. 
In particular, loops are not contained in any path.

The Tutte polynomial $T(G;x,y)$ of a graph $G$ can be defined via the rank function of the graph. 
For any $A\subseteq E(G)$, let $\rho_G(A)=|V(G)|-\kappa(G|A)$, where $G|A=(V(G),A)$ is the spanning subgraph determined by the edge set $A$, 
and $\kappa(G|A)$ denotes the number of connected components of $G|A$. 
Let $G=(V(G),E(G))$ be a graph. The Tutte polynomial of $G$ is defined as
\begin{equation}\label{eq:rooted-tutte}
T(G;x,y)=\sum_{A\subseteq E(G)}(x-1)^{\rho_G(E(G))-\rho_G(A)}(y-1)^{|A|-\rho_G(A)}.
\end{equation}
The Tutte polynomial is a very important graph invariant, which not only has wide applications in graph theory, 
but also plays a significant role in other fields such as knot theory and statistical physics.
For more research progress and applications of the Tutte polynomial, we refer to \cite{Brylawski,EllisMonaghanMoffatt2022,Oxley,Vertigan,Welsh}.

The Tutte polynomial of graphs is one of the most important polynomial invariants in graph theory and matroid theory. Since its rank-nullity expression depends only on the corresponding rank function, it can be naturally extended to combinatorial structures beyond matroids, in particular to greedoids.

\begin{definition}\rm(\cite{KnappNoble2025})
\label{def:greedoid}
A greedoid $\Gamma$ is an ordered pair $(E,\mathcal F)$, where $E$ is a finite set and $\mathcal F$ is a nonempty family of subsets of $E$ satisfying the following axioms:
\begin{enumerate}
 \item[\textup{(i)}] $\varnothing\in\mathcal F$;
 \item[\textup{(ii)}] for any $F,F'\in\mathcal F$, if $|F'|<|F|$, then there exists $x\in F\setminus F'$ such that $F'\cup\{x\}\in\mathcal F$.
\end{enumerate}
\end{definition}

The set $E$ is the ground set of $\Gamma$, and the elements of $\mathcal{F}$ are the feasible sets of $\Gamma$. 
The rank $\rho_{\Gamma}$ of a subset $A$ of $E$ is defined as $\rho_{\Gamma}(A)=\max\{|A'|:A'\subseteq A,\ A'\in\mathcal F\}$,
and we set $\rho(\Gamma)=\rho_{\Gamma}(E)$.

A rooted graph is a graph with a distinguished vertex called the root. 
For a rooted graph $G$ with vertex set $V(G)$, edge set $E(G)$, and root $r(G)$, we write $G=(V(G),E(G),r(G))$. 
For $A\subseteq E(G)$, let the rooted spanning subgraph $G|A=(V(G),A,r(G))$ be obtained from $G$ by deleting all edges in $E(G)\setminus A$ and keeping all vertices,
the connected component of $G|A$ containing the root $r(G)$ is called the root component of $G|A$, denoted by $C_G(A)$. 
If $C_G(A)$ is a tree and contains all edges of $A$, then $A$ is called a feasible set of $G$. 
Let $\mathcal F(G)=\{A\subseteq E(G): A \text{ is a feasible set of } G\}$. 
The greedoid induced by the rooted graph $G$ is denoted by $\Gamma(G)=(E(G),\mathcal F(G))$. 
For $A\subseteq E(G)$, the rank function is defined as $\rho_{\Gamma(G)}(A)=\max\{|A'|: A'\subseteq A,\ A'\in\mathcal F(G)\}$.
If $A\in\mathcal F(G)$ and $\rho(\Gamma(G))=|A|$, then $A$ is called a basis of $\Gamma(G)$. 

Generalizing the definition of the Tutte polynomial for matroids, McMahon and Gordon defined the Tutte polynomial for greedoids in \cite{GordonMcMahon1989}. 
For a greedoid $\Gamma$ with ground set $E$ and rank function $\rho_\Gamma$, its Tutte polynomial is defined as
\begin{equation}\label{eq:rooted-greedoid-tutte}
T(\Gamma;x,y) = \sum_{A \subseteq E} (x - 1)^{\rho(\Gamma)-\rho_\Gamma(A)} (y - 1)^{|A|-\rho_\Gamma(A)}.
\end{equation}
From the above definition, setting $x=y=2$, we obtain
$T(\Gamma;2,2)=\sum_{A\subseteq E}
(2-1)^{\rho(\Gamma)-\rho_\Gamma(A)}
(2-1)^{|A|-\rho_\Gamma(A)}
=\sum_{A\subseteq E}1
=2^{|E|}>0$.
Thus, $T(\Gamma;x,y)\not\equiv0$. When $\Gamma$ is a matroid, this definition reduces to the standard definition of the matroid Tutte polynomial. 
For a rooted graph $G$, the Tutte polynomial of the greedoid induced by $G$ is denoted by $T(\Gamma(G);x,y)$. 
For related research on the Tutte polynomial of matroids and greedoids, 
we refer to \cite{BauerEtAl2026,Brylawski,Crapo1969,EllisMonaghanMoffatt2022, Guan2023, GeelenZhou2008, KnappNoble2025,Zhou2004}.

For a long time, the factorization properties of the greedoid Tutte polynomial of rooted graphs 
have been studied to recover the combinatorial structure of the corresponding rooted graphs. 
In the original variables $t$ and $z$, McMahon \cite{McMahon1993} proved that for any rooted graph, 
the multiplicity of the factor $z+1$ in its greedoid polynomial is exactly equal to the number of greedoid loops in the rooted graph. 
Under the variable substitution $x=t+1$ and $y=z+1$, this result completely characterizes, 
from a graph-theoretic perspective, the highest power of $y$ that divides $T(\Gamma(G);x,y)$. 
However, the graph-theoretic structure corresponding to the factor $x$ has not been precisely characterized. 
Brylawski and Oxley \cite{Brylawski} only proved that for the matroid Tutte polynomial, 
the highest power of $x$ dividing the polynomial is equal to the number of coloops in the matroid. 
This result does not hold for greedoids. Nevertheless, for greedoids, the highest power of $x$ dividing $T(\Gamma(G);x,y)$ may still admit a combinatorial interpretation. 
In particular, Gordon and McMahon proposed the following conjecture in \cite{GordonMcMahon2022}:

\begin{conjecture}\rm(\cite{GordonMcMahon2022})
\label{conj:3352}
Let $G=(V(G),E(G),r(G))$ be a rooted graph, $r=r(G)$, $L_r(G)=\{v\in V(G)\setminus\{r\}: rv\in E(G),\ d_G(v)=1\}$, and $\ell_r(G)=|L_r(G)|$. Then
\begin{eqnarray*}
 \operatorname{ord}_{x}T(\Gamma(G);x,y)=\ell_r(G),
\end{eqnarray*}
where $\operatorname{ord}_{x}T(\Gamma(G);x,y)=\max\{a\in\mathbb Z_{\geq0}:x^a\mid T(\Gamma(G);x,y)\}$.
\end{conjecture}

We find that the above conjecture has counterexamples for multigraphs:
\begin{example}
\label{exa:multigraph-counterexample}
Take an arbitrary integer $q \geq 1$, let $[q]=\{1,\ldots,q\}$, and take a nonempty set $I\subseteq[q]$. For each $i\in I$, take an arbitrary positive integer $t_i$, and let $t=\sum_{i\in I}t_i$. Let $E_0=\{e_i=r(G)v_i:1\leq i\leq q\}$, and
$E_1=\{C_1^{ij}:i\in I,\ 1\leq j\leq t_i,\ C_1^{ij}\text{ is a loop incident to }v_i\}$.
Let $G=(V(G),E(G),r(G))$ be a connected rooted multigraph, where $r=r(G)$, $V(G)=\{r,v_1,\ldots,v_q\}$,
$E(G)=E_0\cup E_1$. In fact, for any $S\subseteq E(G)$, $S$ can be uniquely written as $S=A\cup B$, where $A\subseteq E_0$, $B\subseteq E_1$.
All loops are not contained in any feasible set, while $A$ is the edge set of a star tree centered at the root $r$. Hence $\rho_{\Gamma(G)}(S)=|A|$, and $\rho(\Gamma(G))=q$.

By the definition formula \eqref{eq:rooted-greedoid-tutte} of the greedoid Tutte polynomial and the binomial theorem, we obtain
\begin{align*}
 T(\Gamma(G);x,y)&=\sum_{A\subseteq E_0}\sum_{B\subseteq E_1}(x-1)^{q-|A|}(y-1)^{|B|}\\
 &=\left(\sum_{A\subseteq E_0}(x-1)^{q-|A|}\right)\left(\sum_{B\subseteq E_1}(y-1)^{|B|}\right)\\
 &=x^q y^t.
\end{align*}
On the other hand, for $i\in I$, each loop contributes $2$ to the degree, so $d_G(v_i)=1+2t_i>1$; for $i\notin I$, we have $d_G(v_i)=1$. Therefore,
$L_r(G)=\{v_i:i\in[q]\setminus I\}$, and $\ell_r(G)=q-|I|<q=\operatorname{ord}_{x}T(\Gamma(G);x,y)$.
Thus $G$ is a counterexample to Conjecture \ref{conj:3352}.
\end{example}

\begin{example}\label{exa:loop-extension-counterexamples}
Take arbitrary integers $n\geq4$ and $q\geq1$, let
$[q]=\{1,\ldots,q\}$, and take a nonempty set $I\subseteq[q]$.
For each $i\in I$, take an arbitrary positive integer $t_i$, and let
$t=\sum_{i\in I}t_i$. Take a path
$P_n=v_1v_2\ldots v_n$, designate $r=v_2$ as the root, and let
$w_1=v_1$. When $q\geq2$, take distinct vertices $w_2,\ldots,w_q$ outside $V(P_n)$.
Define the edge sets
$E_0=\{e_j=v_jv_{j+1}:2\leq j\leq n-1\}$,
$E_1=\{f_i=rw_i:1\leq i\leq q\}$,
$E_2=\{C_1^{ij}:i\in I,\ 1\leq j\leq t_i,\ C_1^{ij}\text{ is a loop incident to }w_i\}$.
Let $G=(V(G),E(G),r(G))$ be a rooted multigraph, where
$r(G)=r=v_2$, $V(G)=\{v_1,v_2,\ldots,v_n\}\cup\{w_2,\ldots,w_q\}$,
$E(G)=E_1\cup E_2\cup E_3$.
When $q=1$, we set $\{w_2,\ldots,w_q\}=\varnothing$.
For any $S\subseteq E(G)$, there is a unique decomposition $S=A\cup B\cup C$,
where $A\subseteq E_0$, $B\subseteq E_1$, $C\subseteq E_2$.

For $A\subseteq E_0$, let $\lambda(A)=\max\bigl\{j\in\{0,1,\ldots,n-2\}:\{e_2,e_3,\ldots,e_{j+1}\}\subseteq A\bigr\}$,
where for $j=0$, we set $\{e_2,e_3,\ldots,e_{j+1}\}=\varnothing$.
In other words, $\lambda(A)$ is the number of consecutive edges from the root $r=v_2$ along the path $v_2v_3\cdots v_n$ that belong to $A$. Then the vertex set of the root component of the rooted spanning subgraph $G|S$ is $V(C_G(S))=\{r\}\cup\{w_i:f_i\in D\}\cup\{v_3,v_4,\ldots,v_{\lambda(A)+2}\}$.
Here the last set is empty when $\lambda(A)=0$. Hence $\rho_{\Gamma(G)}(S)=|D|+\lambda(A),\rho(\Gamma(G))=n+q-2$.

We now compute the rank-generating expression corresponding to the long path root component. If $\lambda(A)=j<n-2$, then $e_2,\ldots,e_{j+1}\in A$, $e_{j+2}\notin A$, and the remaining $n-3-j$ edges can be chosen arbitrarily. When $\lambda(A)=n-2$, we must have $A=E_0$. Therefore, by the binomial theorem, we have
\begin{align*}
 R_{n-2}(x,y)&=\sum_{A\subseteq E_P}(x-1)^{n-2-\lambda(A)}(y-1)^{|A|-\lambda(A)}\\
 &=1+\sum_{j=0}^{n-3}(x-1)^{n-2-j}
 \sum_{s=0}^{n-3-j}\binom{n-3-j}{s}(y-1)^s\\
 &=1+\sum_{j=0}^{n-3}(x-1)^{n-2-j}y^{n-3-j}\\
 &=1+\sum_{h=1}^{n-2}(x-1)^hy^{h-1}.
\end{align*}
By the definition formula \eqref{eq:rooted-greedoid-tutte} of the greedoid Tutte polynomial, together with the above rank equalities and the binomial theorem, we obtain
\begin{align*}
 T(\Gamma(G);x,y) &=\sum_{A\subseteq E_0}\sum_{B\subseteq E_1}\sum_{C\subseteq E_2}(x-1)^{n+q-2-|B|-\lambda(A)}(y-1)^{|A|+|C|-\lambda(A)}\\
 &=\left(\sum_{B\subseteq E_S}(x-1)^{q-|B|}\right)\left(\sum_{C\subseteq E_L}(y-1)^{|C|}\right) R_{n-2}(x,y)\\
 &=x^qy^tR_{n-2}(x,y)\\
 &=x^qy^t\left(1+\sum_{h=1}^{n-2}(x-1)^hy^{h-1}\right).
\end{align*}
Since $n\geq4$, we have $R_{n-2}(0,y)=y-y^2+\ldots+(-1)^{n-2}y^{n-3}\not\equiv0$.
Thus $\operatorname{ord}_{x}T(\Gamma(G);x,y)=q$.

On the other hand, for $i\in I$, each loop contributes $2$ to the degree of the vertex, so
$d_G(w_i)=1+2t_i>1$; for
$i\in[q]\setminus I$, we have $d_G(w_i)=1$. The other leaf vertex
$v_n$ of the path is not adjacent to the root $r=v_2$, while the remaining vertices are not
root-adjacent leaves. Therefore $L_r(G)=\{w_i:i\in[q]\setminus I\}$,
so $\ell_r(G)=q-|I|<q=\operatorname{ord}_{x}T(\Gamma(G);x,y)$.
Thus $G$ is a counterexample to Conjecture \ref{conj:3352}.
\end{example}

Therefore, in this paper, we mainly prove that this conjecture holds for simple rooted graphs, and present the following main theorem.

\begin{theorem}
\label{thm:main}
Let $G=(V(G),E(G),r(G))$ be a simple rooted graph, $r=r(G)$, $L_r(G)=\{v\in V(G)\setminus\{r\}: rv\in E(G),\ d_G(v)=1\}$, and $\ell_r(G)=|L_r(G)|$. Then
\begin{eqnarray*}
 \operatorname{ord}_{x}T(\Gamma(G);x,y)=\ell_r(G),
\end{eqnarray*}
where $\operatorname{ord}_{x}T(\Gamma(G);x,y)=\max\{a\in\mathbb Z_{\geq0}:x^a\mid T(\Gamma(G);x,y)\}$.
\end{theorem}

This paper mainly focuses on the research related to Conjecture \ref{conj:3352}, and the organization of the paper is as follows. In Section 2, we introduce some necessary lemmas. In Section 3, we prove that for simple rooted graphs, the highest power of $x$ dividing $T(\Gamma(G);x,y)$ is equal to the number of leaf vertices adjacent to the root.

\section{Preliminaries}\label{sec:prelim}

In this section, we introduce some lemmas needed for the proof of the main theorem.

\begin{lemma}
\label{lem:rank}
For any $A\subseteq E(G)$, we have
\begin{eqnarray*}
 \rho_{\Gamma(G)}(A)=|V(C_G(A))|-1.
\end{eqnarray*}
In particular, if $G$ is connected and $|V(G)|=n$, then
$\rho(\Gamma(G))=n-1$.
\end{lemma}

\begin{proof}
Take any $A'\subseteq A$ with $A'\in\mathcal F(G)$. By the definition of feasible sets, the root component $C_G(A')$ of $G|A'$ is a tree containing all edges of $A'$. Therefore, $|A'|=|V(C_G(A'))|-1$. Since $A'\subseteq A$, every vertex reachable from the root $r(G)$ by edges in $A'$ is also reachable from $r(G)$ by edges in $A$, so $V(C_G(A'))\subseteq V(C_G(A))$. Hence $|A'|\leq |V(C_G(A))|-1$. Taking the maximum over all $A'\subseteq A$ with $A'\in\mathcal F(G)$, we obtain $\rho_{\Gamma(G)}(A)\leq |V(C_G(A))|-1$.

Conversely, take a spanning tree $T$ of the connected graph $C_G(A)$, and let $A_T=E(T)$. Since $T$ is a spanning tree of $C_G(A)$, we have $A_T\subseteq A$, and $|A_T|=|V(C_G(A))|-1$. In the rooted spanning subgraph $G|A_T$, the connected component containing the root $r(G)$ is exactly $T$, and $T$ contains all edges of $A_T$. Hence $A_T\in\mathcal F(G)$. By the definition of $\rho_{\Gamma(G)}(A)$, we have $\rho_{\Gamma(G)}(A)\geq |A_T|=|V(C_G(A))|-1$. Combining the two inequalities above, we obtain $\rho_{\Gamma(G)}(A)=|V(C_G(A))|-1$.

In particular, if $G$ is connected and $|V(G)|=n$, then $C_G(E(G))=G$. Therefore,
\begin{eqnarray*}
\rho(\Gamma(G))=\rho_{\Gamma(G)}(E(G))=|V(C_G(E(G)))|-1=|V(G)|-1=n-1.
\end{eqnarray*}
\end{proof}

\begin{definition}\label{def:nonzero}
For a nonzero polynomial $P(x,y)\in\mathbb Z[x,y]$, define
$\operatorname{ord}_{x}P
=\max\{a\in\mathbb Z_{\geq0}:x^a\mid P(x,y)\}$.
Equivalently, $\operatorname{ord}_{x}P=a$ if and only if
$P=x^aQ$ and $Q(0,y)\not\equiv0$.
\end{definition}

\begin{lemma}\label{lem:outside}
Let $G_r=C_G(E(G))$ be the connected component of $G$ containing the root $r(G)$, and let $q=|E(G)\setminus E(G_r)|$.
Then
\begin{equation}\label{eq:outside}
 T(\Gamma(G);x,y)=y^qT(\Gamma(G_r);x,y),
\end{equation}
where the root of $G_r$ is still $r(G)$. Therefore,
$\operatorname{ord}_{x}T(\Gamma(G);x,y)
=\operatorname{ord}_{x}T(\Gamma(G_r);x,y)$, and
$\ell_r(G)=\ell_r(G_r)$.
\end{lemma}

To prove Lemma \ref{lem:outside}, we first give some lemmas needed in the proof.

\begin{lemma}\rm(\cite{KnappNoble2025})\label{lem:rank-union}
Let $\Gamma=(E,\mathcal F)$ be a greedoid with rank function
$\rho_\Gamma$, and let $A,B\subseteq E$.
If for every $b\in B$, $\rho_\Gamma(A\cup\{b\})=\rho_\Gamma(A)$, then
\begin{eqnarray*}
 \rho_\Gamma(A\cup B)=\rho_\Gamma(A).
\end{eqnarray*}
\end{lemma}

Let $\Gamma=(E,\mathcal F)$ be a greedoid, and let $e\in E$.
If $e$ is not contained in any feasible set of $\Gamma$, i.e.,
$\{F\in\mathcal F:e\in F\}=\varnothing$, then $e$
is called a loop of $\Gamma$. In particular, for a rooted graph
$G=(V(G),E(G),r(G))$, an edge $e\in E(G)$ is
a loop of $\Gamma(G)$ if and only if for every
$F\in\mathcal F(G)$, we have $e\notin F$.

\begin{lemma}\label{lem:greedoid-loop}
Let $G=(V(G),E(G),r(G))$ be a rooted graph, and let $e\in E(G)$.
Then $e$ is a loop of $\Gamma(G)$ if and only if there is no path in $G$ starting from the root $r(G)$ and containing the edge $e$.
\end{lemma}

\begin{proof}
First suppose that $e$ is a loop of $\Gamma(G)$. Assume to the contrary that there is a path $P$ in $G$ starting from the root $r(G)$ and containing the edge $e$. Let $A=E(P)$. Since $P$ is a tree and the root component of the rooted spanning subgraph $G|A$ is exactly $P$, we have $C_G(A)=P$ containing all edges of $A$. By the definition of feasible sets, $A\in\mathcal F(G)$. On the other hand, $e\in E(P)=A$, which contradicts the fact that $e$ is not contained in any feasible set. Therefore, there is no path in $G$ starting from $r(G)$ and containing $e$.

Conversely, suppose that $e$ is not a loop of $\Gamma(G)$. Then there exists $F\in\mathcal F(G)$ such that $e\in F$. By the definition of feasible sets, $C_G(F)$ is a tree containing all edges of $F$, hence $e\in E(C_G(F))$. Since a tree contains no loops, the two endpoints of $e$ are distinct. Let $e=uv$. Among the two connected components of $C_G(F)-e$, exactly one does not contain the root $r(G)$; assume without loss of generality that this component contains the vertex $v$. Then the unique path in $C_G(F)$ from $r(G)$ to $v$ must pass through the edge $e$, thus yielding a path starting from $r(G)$ and containing $e$, contradicting the assumption. Hence $e$ is a loop of $\Gamma(G)$.
\end{proof}

\begin{lemma}\label{lem:ord-product}
If $P(x,y),Q(x,y)\in\mathbb Z[x,y]\setminus\{0\}$, then
\begin{eqnarray*}
 \operatorname{ord}_{x}\bigl(P(x,y)Q(x,y)\bigr)=\operatorname{ord}_{x}P(x,y)+\operatorname{ord}_{x}Q(x,y).
\end{eqnarray*}
\end{lemma}

\begin{proof}
Let $a=\operatorname{ord}_{x}P$,
$b=\operatorname{ord}_{x}Q$. By Definition \ref{def:nonzero}, there exist
$P_0,Q_0\in\mathbb Z[x,y]$ such that
$P=x^aP_0$, $Q=x^bQ_0$,
$P_0(0,y)\not\equiv0$ and $Q_0(0,y)\not\equiv0$.
Then $PQ=x^{a+b}P_0Q_0$. Since $\mathbb Z[y]$ is an integral domain,
we have $(P_0Q_0)(0,y)=P_0(0,y)Q_0(0,y)\not\equiv0$. Hence
$\operatorname{ord}_{x}(PQ)=a+b$.
\end{proof}

\begin{proof}[\textbf{Proof of Lemma \ref{lem:outside}}]
Let $r=r(G)$ and $L=E(G)\setminus E(G_r)$. By the definition of the root component, we have $E(G)=E(G_r)\cup L$, $|L|=q$.

First, we prove the equality needed in the proof of this lemma. Take any $e\in L$. Since $e$ lies in some connected component of $G$ not containing the root $r$, there is no path in $G$ starting from $r$ and containing $e$. By Lemma \ref{lem:greedoid-loop}, we know that $e$ is a loop of $\Gamma(G)$.

Now take any $A\subseteq E(G_r)$ and $B\subseteq L$. For each $b\in B$, since $b$ is a loop, $b$ is not contained in any feasible set of $\Gamma(G)$. Therefore,
$\{F\in\mathcal F(G):F\subseteq A\cup\{b\}\}
=\{F\in\mathcal F(G):F\subseteq A\}$.
By the definition of the rank function $\rho_{\Gamma(G)}$, we have
$\rho_{\Gamma(G)}(A\cup\{b\})=\rho_{\Gamma(G)}(A)$.
Then by Lemma \ref{lem:rank-union}, we obtain
\begin{equation}\label{eq:rho}
\rho_{\Gamma(G)}(A\cup B)=\rho_{\Gamma(G)}(A).
\end{equation}

On the other hand, since $A\subseteq E(G_r)$, the rooted spanning subgraphs $G|A$ and $G_r|A$ have the same root component, i.e.,
$C_G(A)=C_{G_r}(A)$. By Lemma \ref{lem:rank}, we have
\begin{equation}\label{eq:gagr}
\rho_{\Gamma(G)}(A)=|V(C_G(A))|-1 =|V(C_{G_r}(A))|-1=\rho_{\Gamma(G_r)}(A).
\end{equation}
Combining equations \eqref{eq:rho} and \eqref{eq:gagr}, we obtain that for any $A\subseteq E(G_r)$ and $B\subseteq L$,
\begin{eqnarray*}
 \rho_{\Gamma(G)}(A\cup B)=\rho_{\Gamma(G_r)}(A).
\end{eqnarray*}
In particular, taking $A=E(G_r)$ and $B=L$ in the above equation, since $E(G)=E(G_r)\cup L$, we have
\begin{eqnarray*}
 \rho(\Gamma(G))=\rho_{\Gamma(G)}(E(G)) =\rho_{\Gamma(G)}(E(G_r)\cup L)=\rho_{\Gamma(G_r)}(E(G_r))=\rho(\Gamma(G_r)).
\end {eqnarray*}

Next, we prove the polynomial identity \eqref{eq:outside}. For any $S\subseteq E(G)$, let
$A=S\cap E(G_r)$, $B=S\cap L$.
Since $E(G)=E(G_r)\cup L$, the edge set $S$ can be uniquely written as $S=A\cup B$, where $A\subseteq E(G_r)$, $B\subseteq L$. Hence, by the rank-generating definition formula \eqref{eq:rooted-greedoid-tutte} of the greedoid Tutte polynomial, the above unique decomposition, $\rho_{\Gamma(G)}(A\cup B)=\rho_{\Gamma(G_r)}(A)$ and $\rho(\Gamma(G))=\rho(\Gamma(G_r))$, we have
\begin{align*}
 T(\Gamma(G);x,y)&=\sum_{A\subseteq E(G_r)}\sum_{B\subseteq L}(x-1)^{\rho(\Gamma(G))-\rho_{\Gamma(G)}(A\cup B)}(y-1)^{|A|+|B|-\rho_{\Gamma(G)}(A\cup B)}\\
 &=\left(\sum_{A\subseteq E(G_r)}(x-1)^{\rho(\Gamma(G_r))-\rho_{\Gamma(G_r)}(A)}(y-1)^{|A|-\rho_{\Gamma(G_r)}(A)} \right)\left(\sum_{B\subseteq L}(y-1)^{|B|}\right)\\
 &=T(\Gamma(G_r);x,y)\sum_{j=0}^{q}\binom{q}{j}(y-1)^j\\
 &=T(\Gamma(G_r);x,y) \bigl(1+(y-1)\bigr)^q\\
 &=y^qT(\Gamma(G_r);x,y).
\end{align*}

Now we prove that $\operatorname{ord}_{x}T(\Gamma(G);x,y)
=\operatorname{ord}_{x}T(\Gamma(G_r);x,y)$. By the definition of the greedoid Tutte polynomial \eqref{eq:rooted-greedoid-tutte}, setting $x=y=2$, we have
\begin{eqnarray*}
 T(\Gamma(G_r);2,2) =\sum_{A\subseteq E(G_r)}(2-1)^{\rho(\Gamma(G_r))-\rho_{\Gamma(G_r)}(A)}(2-1)^{|A|-\rho_{\Gamma(G_r)}(A)}=\sum_{A\subseteq E(G_r)}1 =2^{|E(G_r)|}>0.
\end{eqnarray*}
Therefore, $T(\Gamma(G_r);x,y)\not\equiv0$, i.e., $T(\Gamma(G_r);x,y)$ is a nonzero polynomial. Since $\operatorname{ord}_{x}(y^q)=0$, by the polynomial identity \eqref{eq:outside} and Lemma \ref{lem:ord-product}, we have
\begin{eqnarray*}
 \operatorname{ord}_{x}T(\Gamma(G);x,y) = \operatorname{ord}_{x}(y^q)+\operatorname{ord}_{x}T(\Gamma(G_r);x,y)=\operatorname{ord}_{x}T(\Gamma(G_r);x,y).
\end{eqnarray*}

Finally, we prove that the number of root-adjacent leaves is invariant. By the definition of the root component, every vertex adjacent to the root $r$ belongs to $V(G_r)$, hence $N_G(r)=N_{G_r}(r)$. Take any $v\in V(G_r)$ and $vw\in E(G)$. Since there is a path in $G_r$ from $r$ to $v$, the subgraph consisting of this path together with the edge $vw$ contains a path from $r$ to $w$. Hence $w$ and $r$ lie in the same connected component of $G$. Therefore, $w\in V(G_r)$ and $vw\in E(G_r)$. This shows that for each $v\in V(G_r)$, we have $d_{G_r}(v)=d_G(v)$. By the definition of $L_r(G)$, we have
\begin{eqnarray*}
\begin{aligned}
 L_r(G)&=\{v\in N_G(r):d_G(v)=1\}\\
 &=\{v\in N_{G_r}(r):d_{G_r}(v)=1\}\\
 &=L_r(G_r).
\end{aligned}
\end{eqnarray*}
Therefore, $\ell_r(G)=|L_r(G)|=|L_r(G_r)|=\ell_r(G_r)$. 

This completes the proof of Lemma \ref{lem:outside}.
\end{proof}

Let $G=(V(G),E(G),r(G))$ be a connected simple rooted graph, $r=r(G)$, $G-r=G[V(G)\setminus\{r\}]$.
Let $\mathcal C(G-r)=\{H_1,\ldots,H_t\}$ be the connected components of $G-r$. For each $i\in\{1,\ldots,t\}$, let $G_i=(G[V(H_i)\cup\{r\}],r)$, and call $G_i$ the $i$-th root component of $G$. Since $H_1,\ldots,H_t$ are all connected components of $G-r$, and the simple graph $G$ has no loops at the root $r$, each edge either connects $r$ to some $H_i$, or its two endpoints belong to the same unique $H_i$. Therefore,
$E(G)=\bigcup_{i=1}^{t}E(G_i)$,
$E(G_i)\cap E(G_j)=\varnothing$, and
$V(G_i)\cap V(G_j)=\{r\}$, where $i\neq j$.
\begin{lemma}\label{lem:product}
For a connected simple rooted graph $G$ with $|V(G)|\geq2$, we have
\begin{equation}
 T(\Gamma(G);x,y)=\prod_{i=1}^{t}T(\Gamma(G_i);x,y). \label{eq:root-product}
\end{equation}
\end{lemma}

To prove Lemma \ref{lem:product}, we first introduce the following lemma.

\begin{lemma}\rm(\cite{GordonMcMahon1989})\label{lem:direct-sum-product}
Let $\Gamma_i=(E_i,\mathcal F_i)$ be greedoids, $i=1,2$, with $E_1\cap E_2=\varnothing$. Let
\begin{eqnarray*}
\Gamma_1\oplus\Gamma_2=bigl(E_1\cup E_2,\{F_1\cup F_2: F_1\in\mathcal F_1,\ F_2\in\mathcal F_2\}\bigr).
\end{eqnarray*}
If $\Gamma=\Gamma_1\oplus\Gamma_2$, then
$T(\Gamma;x,y)=T(\Gamma_1;x,y)T(\Gamma_2;x,y)$.
\end{lemma}

\begin{proof}[\textbf{Proof of Lemma \ref{lem:product}}]
For each $i\in\{1,\ldots,t\}$, let $E_i=E(G_i)$, $\mathcal F_i=\mathcal F(G_i)$.
Take any $A\subseteq E(G)$, and let $A_i=A\cap E_i$. Since
$E(G)=\bigcup_{i=1}^{t}E_i$, we have
$A=\bigcup_{i=1}^{t}A_i$. We now prove that $A\in\mathcal F(G)$ if and only if for every $i\in\{1,\ldots,t\}$, $A_i\in\mathcal F_i$.

If $A\in\mathcal F(G)$, then by the definition of feasible sets, $C_G(A)$ is a tree rooted at $r$ with edge set $A$. Fix $i\in\{1,\ldots,t\}$, 
and define the graph $T_i$ with vertex set and edge set $V(T_i)=V(C_G(A))\cap V(G_i)$ and $E(T_i)=A_i$, respectively.
Since $A_i\subseteq A=E(C_G(A))$, $A_i\subseteq E(G_i)$, and $T_i$ is a subgraph of the tree $C_G(A)$, it follows that $T_i$ is acyclic.

We now prove that $T_i$ is connected. It suffices to show that every vertex $v\in V(T_i)$ is connected to the root $r$ in $T_i$. Take any $v\in V(T_i)$. If $v=r$, then by the definition of the root component, we have $r\in V(C_G(A))$, and by the definition of the root component $G_i$, we have $r\in V(G_i)$, so $r\in V(C_G(A))\cap V(G_i)=V(T_i)$. Let $P_r=(\{r\},\varnothing)$, then $P_r$ is a trivial path of length $0$ from $r$ to itself in $T_i$. Therefore, the vertex $v=r$ is connected to the root $r$ in $T_i$. 

If $v\neq r$, then by $V(G_i)=V(H_i)\cup\{r\}$, we have $v\in V(H_i)$. Let $P_v$ be the unique path in the tree $C_G(A)$ from $r$ to $v$. The graph $P_v-r$ is a connected subgraph of $G-r$ containing $v$, and $H_i$ is the connected component of $G-r$ containing $v$, so $P_v-r\subseteq H_i$, hence $P_v\subseteq G_i$. Since $P_v$ is a subgraph of $C_G(A)$ and $E(C_G(A))=A$, we have $E(P_v)\subseteq A$. Since $P_v\subseteq G_i$, we have $E(P_v)\subseteq E(G_i)$. Therefore, $E(P_v)\subseteq A\cap E(G_i)$. By the definitions of $E_i=E(G_i)$ and $A_i=A\cap E_i$, we have $A\cap E(G_i)=A_i$, and by the definition of $T_i$, we have $A_i=E(T_i)$. In summary, $E(P_v)\subseteq A\cap E(G_i)=A_i=E(T_i)$. Moreover, according to $P_v\subseteq C_G(A)$ and $P_v\subseteq G_i$, we have $V(P_v)\subseteq V(C_G(A))\cap V(G_i)=V(T_i)$. Combining the above edge inclusion, we have $P_v\subseteq T_i$. Hence $P_v$ is a path in $T_i$ connecting $r$ and $v$, i.e., $v$ is connected to $r$ in $T_i$. By the arbitrariness of $v$, every vertex in $V(T_i)$ belongs to the connected component of $T_i$ containing the root $r$, so $T_i$ has only one connected component, i.e., $T_i$ is connected. Since $T_i$ is acyclic, $T_i$ is a tree. Since $E(T_i)=A_i$, every edge of the rooted spanning subgraph $G_i|A_i$ belongs to the connected component $T_i$ containing the root $r$. Hence $C_{G_i}(A_i)=T_i$, and thus $A_i\in\mathcal F_i$.

Conversely, if $A_i\in\mathcal F_i$ for every $i\in\{1,\ldots,t\}$, then by the definition of feasible sets, $T_i=C_{G_i}(A_i)$ is a tree containing the root $r$ with edge set $A_i$. Let $T=\bigcup_{i=1}^{t}T_i$. Take any $v\in V(T)$, then there exists $i\in\{1,\ldots,t\}$ such that $v\in V(T_i)$. Since $T_i$ is connected and $r\in V(T_i)$, there is a path in $T_i$ connecting $r$ and $v$, which is also a subgraph of $T$. Therefore, every vertex in $V(T)$ is connected to $r$ in $T$, so $T$ is connected.

We now prove that $T$ is acyclic. Suppose to the contrary that $T$ contains a cycle $C$. We consider two cases according to whether $r\in V(C)$.

\textbf{Case 1}: If $r\notin V(C)$.

Then $C$ is a connected subgraph of $G-r$. Since $H_1,\ldots,H_t$ are all connected components of $G-r$, there exists a unique $i\in\{1,\ldots,t\}$ such that $C\subseteq H_i$. Hence $C\subseteq G_i$. For any $e\in E(C)$, since $C\subseteq T$, there exists $j$ such that $e\in E(T_j)=A_j$. Since both endpoints of $e$ belong to $V(H_i)$, and distinct $H_i$'s are pairwise disjoint, we must have $j=i$. Hence $E(C)\subseteq E(T_i)$, and consequently $C\subseteq T_i$, contradicting the fact that $T_i$ is acyclic.

\textbf{Case 2}: If $r\in V(C)$.

Then after deleting the vertex $r$ and the edges incident to $r$ from $C$, the resulting graph $C-r$ is a connected subgraph of $G-r$, so there exists a unique $i\in\{1,\ldots,t\}$ such that $C-r\subseteq H_i$. Every edge in $C$ either has both endpoints in $V(H_i)$, or has one endpoint $r$ and the other in $V(H_i)$; hence $C\subseteq G_i$. Since $C\subseteq T$ and distinct root branches have no common vertices other than $r$, we have $E(C)\subseteq E(T_i)$, so $C\subseteq T_i$, again contradicting the fact that $T_i$ is acyclic.

Both cases lead to contradictions, so $T$ is acyclic. Since $T$ is connected, $T$ is a tree. Moreover, since $E(T_i)=A_i$ and $A=\bigcup_{i=1}^{t}A_i$, we have
$E(T)=\bigcup_{i=1}^{t}E(T_i)=\bigcup_{i=1}^{t}A_i=A$.
Thus $T$ is a connected subgraph of the rooted spanning subgraph $G|A$ containing the root $r$ with edge set $A$. Since $E(G|A)=A=E(T)$, there is no edge in $G|A$ that belongs to the root component but not to $T$; and every non-root vertex in the root component lies on a path starting from $r$ with all edges in $A$, so this vertex also belongs to $V(T)$. Therefore $T=C_G(A)$. Since $T$ is a tree and $E(T)=A$, by the definition of feasible sets, we have $A\in\mathcal F(G)$. The equivalence is proved.

Thus
\begin{eqnarray*}
 \mathcal F(G)=\left\{ F_1\cup\cdots\cup F_t:F_i\in\mathcal F_i\ (1\leq i\leq t)\right\},
\end{eqnarray*}
i.e.,
\begin{eqnarray*}
 \Gamma(G)=\Gamma(G_1)\oplus\cdots\oplus\Gamma(G_t).
\end{eqnarray*}
By Lemma \ref{lem:direct-sum-product}, we have
\begin{eqnarray*}
 T(\Gamma(G);x,y) = T\left(\oplus_{i=1}^{t}\Gamma(G_i);x,y\right) = \prod_{i=1}^{t}T(\Gamma(G_i);x,y).
\end{eqnarray*}
\end{proof}

When $G$ is $K_2$ rooted at either endpoint, by direct computation from \eqref{eq:rooted-greedoid-tutte}, we have $T(\Gamma(K_2);x,y)=(x-1)+1=x$.
Since $G$ is connected and simple, for each $i$,
\begin{eqnarray*}
 G_i\cong K_2
 \quad \text{if and only if}\quad
 V(H_i)=\{v_i\},\ rv_i\in E(G),\ d_G(v_i)=1.
\end{eqnarray*}
Therefore, the single-edge root components $G_i\cong K_2$ are in one-to-one correspondence with the leaf vertices in $L_r(G)$.

For a connected simple rooted graph $G=(V(G),E(G),r(G))$,
let $H=G-r(G)$ be a connected graph with $|V(H)|\geq1$, let $s=|V(H)|$, $m=|E(H)|$, and $k=d_G(r(G))$.
Since $H$ is connected, we have $s\geq1$. Since $G$ is connected, the root $r(G)$ is adjacent to at least one vertex in $V(H)$, so $k\geq1$. 
Since $G$ is a simple graph, there is at most one edge between $r(G)$ and each vertex in $V(H)$, so $k\leq s$. Therefore, $1\leq k\leq s$.

For $U\subseteq V(H)$, $J_U=G[\{r(G)\}\cup U]$ denotes the subgraph of $G$ induced by the vertex set $\{r(G)\}\cup U$, 
and define the family of root-connected vertex sets of $G$:
\begin{eqnarray*}
 \mathcal A(G)=\{U\subseteq V(H):J_U\text{ is connected}\}.
\end{eqnarray*}
Here the elements of $\mathcal A(G)$ are subsets of vertices of $V(H)$. Since $G$ is a simple graph,
$J_\varnothing=G[\{r(G)\}]$ satisfies $V(J_\varnothing)=\{r(G)\}$ and $E(J_\varnothing)=\varnothing$,
hence $J_\varnothing\cong K_1$. By the usual convention that a single-vertex graph is connected, we have $\varnothing\in\mathcal A(G)$.

For two disjoint vertex subsets $X,Y\subseteq V(H)$ of the graph $H$, let $E_H(X,Y)=\{xy\in E(H):x\in X,\ y\in Y\}$.
For $U\subseteq V(H)$, let $V(H)\setminus U$ denote the complement of $U$ in $V(H)$, and define
\begin{eqnarray*}
 \delta_H(U)=E_H(U,V(H)\setminus U)=\{uv\in E(H):u\in U,\ v\in V(H)\setminus U\}.
\end{eqnarray*}
The set $\delta_H(U)$ is called the edge cut determined by $(U,V(H)\setminus U)$, and $|\delta_H(U)|$ denotes the number of edges in this edge cut.
Similarly, $E_G(\{r(G)\},U)=\{r(G)u\in E(G):u\in U\}$ is the set of all edges in $G$ connecting the root $r(G)$ to vertices in $U$; let $a_G(U)=|E_G(\{r(G)\},U)|$.
Thus $a_G(U)$ is the number of edges between the root $r(G)$ and the vertex set $U$. Since $G$ is a simple graph, for each $u\in U$, there is at most one edge between $r(G)$ and $u$, therefore, $0\leq a_G(U)\leq |U|$.

\begin{lemma}\label{lem:expansion}
Let $G=(V(G),E(G),r(G))$ be a connected simple rooted graph, and assume that $H=G-r(G)$ is connected with $|V(H)|\geq1$. Let $s=|V(H)|$, $m=|E(H)|$, and for $U\subseteq V(H)$, let $J_U=G[\{r(G)\}\cup U]$, $\mathcal A(G)=\{U\subseteq V(H):J_U\text{ is connected}\}$. Then
\begin{equation}\displaystyle
 T(\Gamma(G);0,y) =\sum_{U\in\mathcal A(G)} (-1)^{s-|U|}y^{|E(H[V(H)\setminus U])|}T(J_U;1,y),
 \label{eq:expansion}
\end{equation}
where $T(J_U;1,y)$ denotes the Tutte polynomial of the graph $J_U$, and $T(J_\varnothing;1,y)=1$.
\end{lemma}

\begin{proof}
By the definition formula \eqref{eq:rooted-greedoid-tutte} of the greedoid Tutte polynomial and setting $x=0$, we have
\begin{eqnarray*}
 T(\Gamma(G);0,y)=\sum_{A\subseteq E(G)} (-1)^{\rho(\Gamma(G))-\rho_{\Gamma(G)}(A)}(y-1)^{|A|-\rho_{\Gamma(G)}(A)}.
\end{eqnarray*}
For any $A\subseteq E(G)$, let $U_A=V(C_G(A))\setminus\{r(G)\}$.
Since $C_G(A)$ is a connected rooted subgraph of $G$, the induced subgraph $J_{U_A}=G[\{r(G)\}\cup U_A]$ is connected, hence $U_A\in\mathcal A(G)$.

Fix $U\in\mathcal A(G)$, and let $\overline U=V(H)\setminus U$. We now show that the edge sets $A\subseteq E(G)$ satisfying $U_A=U$ are in one-to-one correspondence with ordered pairs $(A_U,A_{\overline U})$ satisfying the following conditions, where $A=A_U\cup A_{\overline U}$:
\begin{enumerate}
 \item[\textup{(i)}] $A_U\subseteq E(J_U)$, and $J_U|A_U$ is connected;
 \item[\textup{(ii)}] $A_{\overline U}\subseteq E(H[\overline U])$.
\end{enumerate}

First, suppose $A\subseteq E(G)$ satisfies $U_A=U$, and define $A_U=A\cap E(J_U)$ and $A_{\overline U}=A\cap E(H[\overline U])$. Since $U_A=V(C_G(A))\setminus\{r(G)\}=U$, we have $V(C_G(A))=\{r(G)\}\cup U$. Take any $e=xy\in A$. If $x\in V(C_G(A))$, then by the definition of the root component $C_G(A)$, there is a path from $r(G)$ to $x$ with edge set contained in $A$. Combining this path with the edge $e$, we see that $y$ is also reachable from $r(G)$ by edges in $A$, hence $y\in V(C_G(A))$. Therefore, there is no edge in $A$ with exactly one endpoint in $\{r(G)\}\cup U$. By the definitions of the edge cut $\delta_H(U)$ and the edge set $E_G(\{r(G)\},\overline U)$, this is equivalent to $A\cap\delta_H(U)=\varnothing$ and $A\cap E_G(\{r(G)\},\overline U)=\varnothing$. Since $V(G)=\{r(G)\}\cup U\cup\overline U$, and these three vertex sets are pairwise disjoint, every edge $e\in A$ has both endpoints either in $\{r(G)\}\cup U$, or both in $\overline U$. In the first case $e\in E(J_U)$, and in the second case $e\in E(H[\overline U])$. Therefore,
\begin{eqnarray*}
 A =\bigl(A\cap E(J_U)\bigr) \cup\bigl(A\cap E(H[\overline U])\bigr) = A_U\cup A_{\overline U}.
\end{eqnarray*}
By definition, $A_U\subseteq E(J_U)$ and $A_{\overline U}\subseteq E(H[\overline U])$, so condition \textup{(ii)} holds. Moreover, the vertex set of $C_G(A)$ is $\{r(G)\}\cup U=V(J_U)$, and its edge set is precisely the set of edges in $A$ whose both endpoints lie in $\{r(G)\}\cup U$; hence $E(C_G(A))=A\cap E(J_U)=A_U$. Thus $J_U|A_U=C_G(A)$. Since the root component $C_G(A)$ is connected, $J_U|A_U$ is connected, so condition \textup{(i)} also holds.

We now prove the uniqueness of the above ordered pair. Suppose $(A'_U,A'_{\overline U})$ also satisfies conditions \textup{(i)} and \textup{(ii)}, and $A=A'_U\cup A'_{\overline U}$. Since $V(J_U)=\{r(G)\}\cup U$, $V(H[\overline U])=\overline U$, and $(\{r(G)\}\cup U)\cap\overline U=\varnothing$, we have $E(J_U)\cap E(H[\overline U])=\varnothing$. Since $A'_U\subseteq E(J_U)$ and $A'_{\overline U}\subseteq E(H[\overline U])$, intersecting both sides of $A=A'_U\cup A'_{\overline U}$ with $E(J_U)$ and $E(H[\overline U])$, respectively, yields $A'_U=A\cap E(J_U)=A_U$ and $A'_{\overline U}=A\cap E(H[\overline U])=A_{\overline U}$. Therefore, for a given edge set $A$, the ordered pair $(A_U,A_{\overline U})$ satisfying conditions \textup{(i)}, \textup{(ii)} and $A=A_U\cup A_{\overline U}$ is unique.

Conversely, suppose the ordered pair $(A_U,A_{\overline U})$ satisfies conditions \textup{(i)} and \textup{(ii)}, and let $A=A_U\cup A_{\overline U}$. By condition \textup{(i)}, $J_U|A_U$ is a connected spanning subgraph with vertex set $\{r(G)\}\cup U$. Hence, for each $u\in U$, there is a path in $J_U|A_U$ connecting $r(G)$ and $u$, and all edges of this path belong to $A_U\subseteq A$. Thus $\{r(G)\}\cup U\subseteq V(C_G(A))$. On the other hand, every edge in $A_U$ has both endpoints in $\{r(G)\}\cup U$, and every edge in $A_{\overline U}$ has both endpoints in $\overline U$. Therefore,
\begin{eqnarray*}
 A\cap\bigl(\delta_H(U)\cup E_G(\{r(G)\},\overline U)\bigr)=\varnothing,
\end{eqnarray*}
i.e., there is no edge in $A$ connecting $\{r(G)\}\cup U$ and $\overline U$. Suppose there exists $w\in\overline U\cap V(C_G(A))$. By the definition of the root component, there is a path in $C_G(A)$ from $r(G)$ to $w$. The edge along this path where it first enters $\overline U$ from $\{r(G)\}\cup U$ must belong to $\delta_H(U)\cup E_G(\{r(G)\},\overline U)$, contradicting $w\in\overline U\cap V(C_G(A))$. Hence $V(C_G(A))\cap\overline U=\varnothing$. Combining this with $\{r(G)\}\cup U\subseteq V(C_G(A))$, we have $V(C_G(A))=\{r(G)\}\cup U$. Since the edges of the root component $C_G(A)$ are precisely those edges of $A$ whose both endpoints lie in $\{r(G)\}\cup U$, we have $E(C_G(A))=A_U$. Therefore, $C_G(A)=J_U|A_U$ and $U_A=V(C_G(A))\setminus\{r(G)\}=U$. In summary, this correspondence is bijective.

By Lemma \ref{lem:rank}, for all $A$ satisfying $U_A=U$, we have $\rho_{\Gamma(G)}(A)=|V(C_G(A))|-1=|U|$. Since $G$ is connected and $|V(G)|=s+1$, by Lemma \ref{lem:rank}, we have $\rho(\Gamma(G))=\rho_{\Gamma(G)}(E(G))=s$. Therefore, the total contribution to $T(\Gamma(G);0,y)$ from all edge sets satisfying $U_A=U$ is
\begin{align*}
 &(-1)^{s-|U|}\sum_{\substack{A_U\subseteq E(J_U)\\J_U|A_U\text{is connected}}}(y-1)^{|A_U|-|U|}\sum_{A_{\overline U}\subseteq E(H[\overline U])}(y-1)^{|A_{\overline U}|}.
\end{align*}
Since $J_U$ is connected and $|V(J_U)|=|U|+1$, by the definition of the rank function, we have $\rho_{J_U}(E(J_U))=|U|$. For $A_U\subseteq E(J_U)$, the equality $\rho_{J_U}(A_U)=|U|$ holds if and only if $J_U|A_U$ is connected. Therefore, by the definition formula \eqref{eq:rooted-tutte} of the Tutte polynomial, we have
\begin{eqnarray*}
 \sum_{\substack{A_U\subseteq E(J_U)\\J_U|A_U\text{is connected}}}(y-1)^{|A_U|-|U|}=T(J_U;1,y).
\end{eqnarray*}
On the other hand, by the binomial theorem, we have
\begin{eqnarray*}
 \sum_{A_{\overline U}\subseteq E(H[\overline U])}(y-1)^{|A_{\overline U}|} = (1+(y-1))^{|E(H[\overline U])|}= y^{|E(H[\overline U])|}.
\end{eqnarray*}
Hence the total contribution for a fixed $U\in\mathcal A(G)$ is equal to
\begin{eqnarray*}
 (-1)^{s-|U|} y^{|E(H[V(H)\setminus U])|} T(J_U;1,y).
\end{eqnarray*}
Finally, summing over $U\in\mathcal A(G)$, we obtain
\begin{eqnarray*}
 T(\Gamma(G);0,y) =\sum_{U\in\mathcal A(G)}(-1)^{s-|U|} y^{|E(H[V(H)\setminus U])|}T(J_U;1,y).
\end{eqnarray*}
This is precisely formula \eqref{eq:expansion}.
\end{proof}

\begin{lemma}\rm(\cite{Guan2023})\label{lem:guan-coefficients}
For a graph $G = (V(G), E(G))$ and an integer $k \geq 0$,
the equality $[y^j] T(G;1, y) \\= \binom{|E(G)|-j-1}{|V(G)|-2}$
holds for all $j$ with $|E(G)| - |V(G)| + 1 - k \leq j \leq |E(G)| - |V(G)| + 1$
if and only if $G$ is $(k+1)$-edge connected.
\end{lemma}

From the above lemma, we obtain the following corollary:

\begin{corollary}\rm\label{cor:leading}
Let $G=(V(G),E(G))$ be a connected loopless graph with $|V(G)|=n$, $|E(G)|=m$, and $d=m-n+1$. Then
\begin{enumerate}
 \item[\textup{(i)}] $\deg_yT(G;1,y)=d$, and the coefficient of $y^d$ in $T(G;1,y)$ is $1$;
 \item[\textup{(ii)}] If $G$ is bridgeless and $d\geq1$, then the coefficient of $y^{d-1}$ in $T(G;1,y)$ is $n-1$.
\end{enumerate}
\end{corollary}

\begin{proof}
For $T(G;1,y)=\sum_{j=0}^{d}a_jy^j$ let $j=d-i$, by Lemma \ref{lem:guan-coefficients}, 
we can obtain that $G$ is $(k+1)$-edge-connected if and only if for every integer $i$ satisfying $0\leq i\leq k$,
\begin{eqnarray}\label{eq:ver}
 a_{d-i}=\binom{n+i-2}{n-2}.
\end{eqnarray}

If $n=1$, then since $G$ is loopless and connected, we have $m=0$, hence $d=0$ and $T(G;1,y)=1$. In this case conclusion \textup{(i)} holds, while the condition $d\geq1$ in conclusion \textup{(ii)} is not satisfied. We now assume $n\geq2$.

Since $G$ is connected, $G$ is $1$-edge-connected, and $d=m-n+1\geq0$. Taking $k=0$ in formula \eqref{eq:ver},
 we have $a_d=\binom{n-2}{n-2}=1$. Since $T(G;1,y)=\sum_{j=0}^{d}a_jy^j$, we have $\deg_yT(G;1,y)\leq d$. Since $a_d=1\neq0$, we obtain $\deg_yT(G;1,y)=d$, and the coefficient of $y^d$ in $T(G;1,y)$ is $1$. This proves \textup{(i)}.

Now suppose $G$ is bridgeless and $d\geq1$. Then deleting any single edge does not destroy connectivity, so $G$ is $2$-edge-connected. Taking $k=1$ in formula \eqref{eq:ver},, we have $a_{d-1}=\binom{n-1}{n-2}=n-1$. Therefore, the coefficient of $y^{d-1}$ in $T(G;1,y)$ is $n-1$. This proves \textup{(ii)}.
\end{proof}

\section{Proof of Theorem \ref{thm:main}}
\label{sec:nonvanishing}

To prove Theorem \ref{thm:main}, we first give an important theorem.

\begin{theorem}\label{thm:nonvanishing}
Let $G=(V(G),E(G),r(G))$ be a simple connected rooted graph such that $G-r(G)$ is connected. Then
\begin{eqnarray*}
 T(\Gamma(G);0,y)\equiv 0 \quad \text{if and only if}\quad G\cong K_2.
\end{eqnarray*}
In particular, $G\not\cong K_2$ if and only if $x\nmid T(\Gamma(G);x,y)$.
\end{theorem}

\begin{proof}
If $G\cong K_2$, then by the definition formula \eqref{eq:rooted-greedoid-tutte} of the greedoid Tutte polynomial, direct computation gives $T(\Gamma(G);x,y)=x$, hence $T(\Gamma(G);0,y)=0$.

Now suppose $G\not\cong K_2$. If $|V(G)|=1$, then by simplicity, $E(G)=\varnothing$, hence $T(\Gamma(G);x,y)=1$ and $T(\Gamma(G);0,y)=1\neq0$, so the conclusion holds. We now assume $|V(G)|\geq2$, and adopt the notation defined in Section 2 before Lemma \ref{lem:expansion}. Since $H$ is connected, we have $s\geq1$. If $s=1$, then $V(G)=\{r(G),v\}$. Since $G$ is a connected simple graph, we must have $E(G)=\{r(G)v\}$, hence $G\cong K_2$, contradicting the assumption. Hence, $s\geq2$.

For any $U\in\mathcal A(G)$, denote the summand corresponding to $U$ in the expansion \eqref{eq:expansion} of Lemma \ref{lem:expansion} by
\begin{equation}\label{eq:sum}
\Phi_U(y)=(-1)^{s-|U|}y^{|E(H[V(H)\setminus U])|}T(J_U;1,y).
\end{equation}
Since $U\in\mathcal A(G)$, by the definition of $\mathcal A(G)$, $J_U$ is connected. Since $J_U$ is an induced subgraph of the simple graph $G$, $J_U$ is loopless. In the graph Tutte polynomial $T(J_U;1,y)$, $J_U$ denotes a simple graph. Here $|V(J_U)|=|U|+1$, so the cycle rank of $J_U$ is $|E(J_U)|-|V(J_U)|+1=|E(J_U)|-|U|$. Therefore, applying Corollary \ref{cor:leading}\textup{(i)} to $J_U$, we have $\deg_yT(J_U;1,y)=|E(J_U)|-|U|$, and the coefficient of $y^{|E(J_U)|-|U|}$ in $T(J_U;1,y)$ is $1$.

By the definitions of $J_U$ and $a_G(U)$, we have $E(J_U)=E(H[U])\cup E_G(\{r(G)\},U)$, hence $|E(J_U)|=|E(H[U])|+a_G(U)$. On the other hand, by the definition of $\delta_H(U)$, the edge set $E(H)$ has the disjoint decomposition $E(H)=E(H[U])\cup E(H[V(H)\setminus U])\cup\delta_H(U)$. Therefore, by the definition of $\Phi_U(y)$ and the above arguments, its $y$-degree is
\begin{align}
 D(U) &=|E(H[V(H)\setminus U])|+|E(J_U)|-|U|\notag\\
 &= |E(H[V(H)\setminus U])|+|E(H[U])|+a_G(U)-|U|\notag\\
 &= m-|\delta_H(U)|+a_G(U)-|U|.\label{eq:degree}
\end{align}
From formula \eqref{eq:degree}, we can write $D(U)=m-|\delta_H(U)|-(|U|-a_G(U))$. In Section 2, we have already proved that $0\leq a_G(U)\leq|U|$, so $|U|-a_G(U)\geq0$.
Also according to $|\delta_H(U)|\geq0$, we have $D(U)\leq m-|\delta_H(U)|\leq m$. Moreover, the monomial factor in the definition of $\Phi_U(y)$ only changes the $y$-degree, and its coefficient is $(-1)^{s-|U|}$. Together with the fact that the coefficient of $y^{|E(J_U)|-|U|}$ in $T(J_U;1,y)$ is $1$, we have that the highest coefficient of $\Phi_U(y)$ is $(-1)^{s-|U|}$.

We now consider the three cases for $U$ in $V(H)$.
When $U=\varnothing$, by the explanation of $J_\varnothing$ before Lemma \ref{lem:expansion} in Section 2, we have $J_\varnothing\cong K_1$, and by Lemma \ref{lem:expansion}, we have $T(J_\varnothing;1,y)=T(K_1;1,y)=1$. Substituting $U=\varnothing$ into formula \eqref{eq:sum}, and noting that $|E(H[V(H)])|=|E(H)|=m$, we have $\Phi_\varnothing(y)=(-1)^sy^m$.
When $\varnothing\subsetneq U\subsetneq V(H)$, if $\delta_H(U)=\varnothing$, then there is no edge in $H$ connecting $U$ and $V(H)\setminus U$, contradicting the connectivity of $H$. Hence $\delta_H(U)\neq\varnothing$, i.e., $|\delta_H(U)|\geq1$. From formula \eqref{eq:degree} and $a_G(U)\leq|U|$, we have $D(U)\leq m-|\delta_H(U)|\leq m-1$.
When $U=V(H)$, by the definitions of $a_G$ and $k$, we have $a_G(V(H))=|E_G(\{r(G)\},V(H))|=d_G(r(G))=k$.
By the definition of the edge cut, we have $\delta_H(V(H))=\varnothing$. Substituting these two equalities into \eqref{eq:degree}, we obtain $D(V(H))=m+k-s$.

Finally, before Lemma \ref{lem:expansion} in Section 2, we have already proved from the connectivity and simplicity of $G$ that $1\leq k\leq s$. Hence either $1\leq k<s$ or $k=s$.
In the latter case, $s$ is either even or odd. So we can consider three cases.

\textbf{Case 1}: $1\leq k<s$.

From formula \eqref{eq:degree}, we have $D(V(H))=m+k-s\leq m-1$. Combined with the upper bound $m-1$ for the degree of terms corresponding to nonempty proper subsets obtained from the same formula, we see that in \eqref{eq:expansion}, only $\Phi_\varnothing(y)$ contains the $y^m$ term. Since $T(J_\varnothing;1,y)=1$ from Lemma \ref{lem:expansion}, the coefficient of $y^m$ in $T(\Gamma(G);0,y)$ is $(-1)^s\neq0$, hence $T(\Gamma(G);0,y)\not\equiv0$.

\textbf{Case 2:} $k=s$, and $s$ is even.

Since $G$ is simple and $k=d_G(r(G))=s=|V(H)|$, the root $r(G)$ is adjacent to every vertex in $V(H)$ by exactly one edge. From formula \eqref{eq:degree} and Lemma \ref{lem:expansion}, only the summands corresponding to $U=\varnothing$ and $U=V(H)$ may have degree $m$. By Lemma \ref{lem:expansion} with $T(J_\varnothing;1,y)=1$, the coefficient of $y^m$ in the empty set term is $(-1)^s$. By Corollary \ref{cor:leading}\textup{(i)} and the definition of $\Phi_{V(H)}(y)$, the coefficient of $y^m$ in the full set term is $(-1)^{s-s}=1$. Since $s$ is even, the coefficient of $y^m$ in $T(\Gamma(G);0,y)$ is $(-1)^s+1=2\neq0$, hence $T(\Gamma(G);0,y)\not\equiv0$.

\textbf{Case 3:} $k=s$, and $s$ is odd.

Since $s\geq2$ and $s$ is odd, we have $s\geq3$.
When $U=\varnothing$, by Lemma \ref{lem:expansion} and $J_\varnothing\cong K_1$, we have $\Phi_\varnothing(y)=(-1)^sy^m$. When $\varnothing\subsetneq U\subsetneq V(H)$, by the connectivity of $H$, we have $|\delta_H(U)|\geq1$. Also by $a_G(U)\leq|U|$ and formula \eqref{eq:degree}, we have $D(U)\leq m-1$, so the corresponding $\Phi_U(y)$ contains no $y^m$ term.
When $U=V(H)$, by $k=s$ and formula \eqref{eq:degree}, we have $D(V(H))=m$.
By Corollary \ref{cor:leading}\textup{(i)}, the definition of $\Phi_U(y)$ and the above arguments, we have the coefficient of $y^m$ in $\Phi_{V(H)}(y)$ is $(-1)^{s-s}=1$. 
Therefore, by formula \eqref{eq:expansion} in Lemma \ref{lem:expansion}, the coefficient of $y^m$ in $T(\Gamma(G);0,y)$ is $(-1)^s+1=0$. 
To prove this case, we consider the coefficient of $y^{m-1}$.

First consider the full set term corresponding to $U=V(H)$. Since $k=s$ and $G$ is simple, 
the root $r(G)$ is adjacent to every vertex in $V(H)$ by exactly one edge. 
If $uv\in E(H)$, then $r(G)u,r(G)v\in E(J_{V(H)})$, so the edge $uv$ belongs to a triangle with vertex set $\{r(G),u,v\}$. 
If $r(G)v\in E(J_{V(H)})\setminus E(H)$, then since $H$ is connected and $s\geq3$, the vertex $v$ has a neighbor $w$ in $H$,
then $vw,r(G)w\in E(J_{V(H)})$, so the edge $r(G)v$ belongs to a triangle with vertex set $\{r(G),v,w\}$. 
Therefore, every edge of $J_{V(H)}$ belongs to some cycle. 
Since an edge is a bridge if and only if it belongs to no cycle, we have that $J_{V(H)}$ is bridgeless.

By the definition of $J_U$, we have $J_{V(H)}=G[\{r(G)\}\cup V(H)]=G$. Since $G$ is connected and simple, $J_{V(H)}$ is connected and loopless. Since $k=s$, we have $|V(J_{V(H)})|=s+1$ and $|E(J_{V(H)})|=m+s$, so the cycle rank of $J_{V(H)}$ is $(m+s)-(s+1)+1=m$. Since $H$ is connected and $s\geq3$, we have $m\geq s-1\geq2$. Thus $J_{V(H)}$ satisfies all the conditions of Corollary \ref{cor:leading}\textup{(ii)}, so the coefficient of $y^{m-1}$ in $T(J_{V(H)};1,y)$ is $|V(J_{V(H)})|-1=s$.
On the other hand, from the expansion \eqref{eq:expansion} in Lemma \ref{lem:expansion} and the definition of $\Phi_U(y)$, we have $\Phi_{V(H)}(y)=T(J_{V(H)};1,y)$. Hence the full set term contributes $s$ to the coefficient of $y^{m-1}$. Since $\Phi_\varnothing(y)=-y^m$, the empty set term contributes nothing to the coefficient of $y^{m-1}$.

Now we consider $\varnothing\subsetneq U\subsetneq V(H)$. Since $k=s$ and $G$ is simple, the root $r(G)$ is adjacent to every vertex in $V(H)$ by exactly one edge, so by the definition of $a_G(U)$, we have $a_G(U)=|U|$. Substituting this into \eqref{eq:degree}, we have $D(U)=m-|\delta_H(U)|$. Therefore, if $|\delta_H(U)|\geq2$, then $D(U)\leq m-2$, and the corresponding summand cannot contribute to the coefficient of $y^{m-1}$; only sets $U$ with $|\delta_H(U)|=1$ need to be considered. For each such $U$, by Corollary \ref{cor:leading}\textup{(i)} and the definition of $\Phi_U(y)$, its contribution to $y^{m-1}$ is $(-1)^{s-|U|}$.

Using the notation $\overline U$ from Section 2. By the definition of the edge cut, $\delta_H(\overline U)=\delta_H(U)$, so $|\delta_H(\overline U)|=1$. Moreover, for any $W\subseteq V(H)$, since $k=s$, $J_W$ contains a star spanning subgraph centered at $r(G)$ with vertex set $\{r(G)\}\cup W$, hence $J_W$ is connected. By the definition of $\mathcal A(G)$, we have $W\in\mathcal A(G)$; in particular, $U,\overline U\in\mathcal A(G)$. Since $\varnothing\subsetneq U\subsetneq V(H)$, the complement map $U\mapsto\overline U$ has no fixed points, so all nonempty proper subsets satisfying $|\delta_H(U)|=1$ are partitioned into pairwise disjoint pairs $\{U,\overline U\}$. For each such pair, since $|\overline U|=s-|U|$ and $s$ is odd, the sum of their contributions to $y^{m-1}$ is
$(-1)^{s-|U|}+(-1)^{s-|\overline U|}
=(-1)^{s-|U|}+(-1)^{|U|}=0$. Therefore, the contributions from all nonempty proper subset terms to $y^{m-1}$ cancel in pairs. Combining the contributions from the empty set term and the full set term, we have that the coefficient of $y^{m-1}$ in $T(\Gamma(G);0,y)$ is $s$, hence $T(\Gamma(G);0,y)\not\equiv0$.

In all three cases above, we have $T(\Gamma(G);0,y)\not\equiv0$. Therefore, when $G\not\cong K_2$, we have $T(\Gamma(G);0,y)\not\equiv0$. 
Finally, for any $P(x,y)\in\mathbb Z[x,y]$, $x\mid P(x,y)$ if and only if $P(0,y)\equiv0$. 
Taking $P(x,y)=T(\Gamma(G);x,y)$, we obtain that $G\not\cong K_2$ if and only if $x\nmid T(\Gamma(G);x,y)$.
\end{proof}

\begin{proof}[\textbf{Proof of Theorem \ref{thm:main}}]
Let $G_r=C_G(E(G))$. By Lemma \ref{lem:outside}, we have 
\begin{eqnarray*}
\operatorname{ord}_{x}T(\Gamma(G);x,y)=\operatorname{ord}_{x}T(\Gamma(G_r);x,y),\quad \ell_r(G)=\ell_r(G_r). 
\end{eqnarray*}
Therefore, it suffices to prove the conclusion for the root component $G_r$. In the following, we rename $G_r$ as $G$; now $G$ is a connected simple rooted graph.

If $V(G)=\{r(G)\}$, then by the connectivity of $G$, we have $E(G)=\varnothing$. By the definition \eqref{eq:rooted-greedoid-tutte} of the greedoid Tutte polynomial, we have $T(\Gamma(G);x,y)=1$, hence $\operatorname{ord}_{x}T(\Gamma(G);x,y)=\ell_r(G)=0$. So the conclusion holds in this case. We now consider the case $|V(G)|\geq2$.

Let $r=r(G)$, and let $\mathcal C(G-r)=\{H_1,\ldots,H_t\}$ be all connected components of $G-r$. For each $i\in\{1,\ldots,t\}$, let
\begin{eqnarray*}
 G_i=\bigl(G[V(H_i)\cup\{r\}],r\bigr).
\end{eqnarray*}

For any $i\in\{1,\ldots,t\}$, $H_i$ is by definition a connected component of $G-r$. Since $G$ is connected, there is at least one edge between $r$ and $V(H_i)$, so $G_i$ is connected. Moreover, $G_i$ is an induced subgraph of the simple graph $G$, and $G_i-r=H_i$ is connected. Therefore, each $G_i$ satisfies all the conditions of Theorem \ref{thm:nonvanishing}.

From the characterization of single-edge root components given after Lemma \ref{lem:product} in Section 2, we have that for any $i\in\{1,\ldots,t\}$, $G_i\cong K_2$ if and only if $V(H_i)=\{v_i\}$, $rv_i\in E(G)$, and $d_G(v_i)=1$. Therefore, the root branches satisfying $G_i\cong K_2$ are in one-to-one correspondence with the leaf vertices in $L_r(G)$, and their number is $\ell_r(G)$.

Fix $i\in\{1,\ldots,t\}$. If $G_i\cong K_2$, then $T(\Gamma(G_i);x,y)=x$, so $\operatorname{ord}_{x}T(\Gamma(G_i);x,y)=1$. If $G_i\not\cong K_2$, then by Theorem \ref{thm:nonvanishing}, we have $T(\Gamma(G_i);0,y)\not\equiv0$. By Definition \ref{def:nonzero}, this implies $\operatorname{ord}_{x}T(\Gamma(G_i);x,y)=0$. Hence, for each $i$, the polynomial $T(\Gamma(G_i);x,y)$ is nonzero.

By formula \eqref{eq:root-product} in Lemma \ref{lem:product}, and applying Lemma \ref{lem:ord-product} repeatedly to the product, we have
\begin{eqnarray*}
 \operatorname{ord}_{x}T(\Gamma(G);x,y)
 =
 \sum_{i=1}^{t}\operatorname{ord}_{x}T(\Gamma(G_i);x,y)
 =
 \bigl|\{i\in\{1,\ldots,t\}:G_i\cong K_2\}\bigr|
 =
 \ell_r(G).
\end{eqnarray*}
This completes the proof of Theorem \ref{thm:main}.
\end{proof}

\section{Conclusion}

Tutte's eponymous polynomial is perhaps the most widely studied two-variable graph and
matroid polynomial due to its many specializations, their vast breadth and the richness ofthe underlying theory.
In this paper, we have focused on Conjecture \ref{conj:3352} and proved that for simple rooted graphs, the highest power of $x$ dividing $T(\Gamma(G);x,y)$ is equal to the number of leaf vertices adjacent to the root. This result shows that, in the greedoid induced by a simple rooted graph, the lowest $x$-degree of the greedoid Tutte polynomial is precisely determined by the number of single-edge root branches at the root, thereby providing a clear graph-theoretic interpretation of this algebraic parameter.

\noindent{\bf Data Availability}\\
{No data were used to support this study.}

\noindent{\bf Acknowledgements:} This research is supported by NSFC (Nos. 12261071, 12471333)  and NSF of Qinghai Province (No. 2025-ZJ-902T).

\end{document}